\documentclass[reqno,12pt]{amsart} 
\usepackage{vmargin}
\setpapersize{A4}
\usepackage{amsmath} 
\usepackage{amsfonts}   
\usepackage{amssymb}      
\usepackage{eufrak}      

\usepackage{amscd}      
\usepackage{amsthm}      
\usepackage{tikz,amsmath}
\usetikzlibrary{arrows}

\usepackage{epsfig}      
\usepackage{amstext}      
\usepackage{graphicx}

\usepackage[all,line,dvips]{xy} 
\CompileMatrices %goer at xy-matricerne koerer hurtigere.

\newcommand{\Ad}{\operatorname{Ad}}

\newcommand{\Tr}{\operatorname{Tr}}

 \newcommand{\supp}{\operatorname{supp}}

\newcommand{\Is}{\operatorname{Is}}

\newcommand{\Sp}{\operatorname{Sp}}

\newcommand{\Var}{\operatorname{Var}}
   \theoremstyle{plain}%default
   \newtheorem{thm}{Theorem}[section]
   
   \newtheorem{lemma}[thm]{Lemma}  
   \newtheorem{cor}[thm]{Corollary}
   \theoremstyle{definition}

   \theoremstyle{remark}
   \newtheorem{obs}[thm]{Observation}
   \newtheorem{remark}[thm]{Remark}
\newtheorem{assert}[thm]{Assertion}

\usepackage{mdframed,xcolor}

\definecolor{mybgcolor}{gray}{0.8}
\definecolor{myframecolor}{rgb}{.647,.129,.149}

\mdfdefinestyle{mystyle}{
  usetwoside=false,
  skipabove=0.6em plus 0.8em minus 0.2em,
  skipbelow=0.6em plus 0.8em minus 0.2em,
  innerleftmargin=.25em,
  innerrightmargin=0.25em,
  innertopmargin=0.25em,
  innerbottommargin=0.25em,
  leftmargin=-.75em,
  rightmargin=-0em,
  topline=false,
  rightline=false,
  bottomline=false,
  leftline=false,
  backgroundcolor=mybgcolor,
  splittopskip=0.75em,
  splitbottomskip=0.25em,
  innerleftmargin=0.5em,
  leftline=true,
  linecolor=myframecolor,
  linewidth=0.25em,
}

\newmdenv[style=mystyle]{important}

\usepackage{lipsum}

   \numberwithin{equation}{section}

        \date{\today}
\title[Phase transition in $O_2$]{Phase
  transition in $O_2$}  
\author{Klaus Thomsen}

\date{\today}

\email{matkt@math.au.dk}
\address{Institut for Matematik, Aarhus University, Ny Munkegade, 8000 Aarhus C, Denmark}

\begin{document}

\maketitle

\begin{center} \emph{I mindet om Uffe Haagerup}
\end{center}

\bigskip

\section{Introduction}\label{sec1}

The Cuntz algebra $O_2$ is the
universal $C^*$-algebra generated by a pair of isometries $V_0,V_1$
such that $V_0V_0^* + V_1V_1^* = 1$, \cite{C}. Let $A = A^* \in O_2$.  For each $t\in
\mathbb R$ the universal property
of $O_2$ guarantees the existence of an endomorphism $\sigma^A_t$  of
$O_2$ such that
\begin{equation}\label{formula}
\sigma^A_t\left(V_i\right) = e^{itA}V_i, \ i = 0,1.
\end{equation}
Under appropriate assumptions on $A$ the family $\sigma^A_t, t \in \mathbb R$, constitutes a continuous
one-parameter family of automorphisms of $O_2$, and it is an
interesting problem to determine the KMS-states for
$\sigma^A$. The main purpose with the present note is to
exhibit an example of such an action where the structure of KMS states
is rich with a relatively dramatic phase transition,
involving an abrupt passage at a certain critical temperature from
uncountably many extremal KMS states to one and
then none. That this can occur in $O_2$ is
perhaps surprising since up to now the one-parameter actions on $O_2$,
or any simple unital Cuntz-Krieger algebra for that matter, for
which it has been possible to determine the structure of KMS states
have all had a unique KMS state, cf. \cite{OP},\cite{E}. In view of
the work of Exel, \cite{E}, it is clear that the possibility of having a
richer structure has to do with the failure of the
Ruelle-Perron-Frobenius theorem for certain potential functions; in
particular, for the function $F$ below. In
this respect, as well as others, the present work is related to the
work of Hofbauer, \cite{H}.

For $u =(i_1,i_2, \cdots , i_n)
\in \{0,1\}^n$, set
$V_u = V_{i_1}V_{i_2} \cdots V_{i_n}$. The elements 
\begin{equation}\label{elements0}
V_{u}V_{u}^*, \  \ u \in \bigcup_n \{0,1\}^n,
\end{equation}
generate an abelian $C^*$-subalgebra of $O_2$ which is
isomorphic to $C\left(\{0,1\}^{\mathbb
    N}\right)$. In the following we identify these two algebras. If the self-adjoint $A$ above comes
from $C\left(\{0,1\}^{\mathbb
    N}\right)$, the formula (\ref{formula}) will define a continuous one-parameter group $\sigma^A$. Therefore, when we define $F : \{0,1\}^{\mathbb N} \to \mathbb R$ such
that 
\begin{equation}\label{F}
F\left((x_i)_{i=1}^{\infty}\right) = \begin{cases} \frac{1}{ \min \left\{i
    : \ x_i = 0\right\} }, & \ (x_i)_{i=1}^{\infty} \neq 1^{\infty} \\ 0,
  & \  (x_i)_{i=1}^{\infty} = 1^{\infty} , \end{cases}
\end{equation}
where $1^{\infty} \in \{0,1\}^{\mathbb N}$ denotes the infinite string
of $1$'s, we have specified a continuous and real-valued function on
$\{0,1\}^{\mathbb N}$ and we can consider the corresponding one-parameter group
$\sigma^F$. Our main result is the following. 

\begin{thm}\label{MAIN1} Let $\beta_0$ be the positive real number for which 
$$
\sum_{k=1}^{\infty} \exp \left( -\beta_0 \sum_{j=1}^k
  \frac{1}{j}\right) \ = \ 1.
$$
There are no $\beta$-KMS states for $\sigma^F$ when $\beta < \beta_0$,
a unique $\beta_0$-KMS state and for $\beta > \beta_0$ the simplex of
$\beta$-KMS states is affinely homeomorphic to the Bauer simplex of
Borel probability measures on the circle $\mathbb T$.
\end{thm}

The factor types of the extremal KMS states are
determined in Section \ref{factortypes}. They are all of type
$I_{\infty}$, except the unique $\beta_0$-KMS state which is of type $III_1$.

\smallskip

\emph{Acknowledgement} I thank Johannes Christensen for helpful remarks and
  corrections to the first versions of this paper.

\section{KMS-states on the $C^*$-algebra of a local homeomorphism}\label{sec2}

The proof of Theorem \ref{MAIN1} and the following description of the KMS-states is
based on the work of Neshveyev, \cite{N}, which offers a very general
description of the KMS-states for a class of one-parameter groups of
automorphisms on the (full) $C^*$-algebra of a locally compact
Hausdorff \'etale groupoid. To describe the relevant (and well-known)
groupoid, note that the shift $\sigma$ acts on $\{0,1\}^{\mathbb N}$
as a local homeomorphism, given by $\sigma(x)_i = x_{i+1}$, and recall
that
there is a general construction of an \'etale groupoid from a local
homeomorphism which was introduced by Renault in \cite{Re1} in the
setting relevant for our main result. Since the work of
Renault the construction of the groupoid and the associated
$C^*$-algebra has been extended by Deaconu
\cite{De}, by Anantharaman-Delaroche, \cite{An}, and finally by
Renault \cite{Re2} again. We shall now describe the
groupoid from \cite{An} and explain what information can be obtained by applying
Neshveyev's work to it.

Let $X$ be a locally compact second countable Hausdorff space and
$\varphi : X \to X$ a local homeomorphism.  
Set
$$
\mathcal G_{\varphi} = \left\{ (x, n-m,y) \in X \times \mathbb Z \times X: \  \varphi^n(x) = \varphi^m(y) \right\} 
$$
which is a groupoid with product $(x,k,y) (y,l,z) = (x,k+l,y)$ and
inversion $(x,k,y)^{-1} = (y,-k,x)$. Sets of the form
$$
\left\{ (x,n-m,y) :  \ \varphi^n(x) = \varphi^m(y), \  x \in W, \ y \in V\right\}, 
$$
for some open subsets $V, W \subseteq X$ constitute a basis for a topology in $\mathcal G_{\varphi}$ which
turns it into a locally compact second countable Hausdorff \'etale
groupoid. As explained by Renault in \cite{Re2} the groupoid $\mathcal
G_{\varphi}$ is amenable and hence the full and reduced groupoid
$C^*$-algebras of $\mathcal G_{\varphi}$ coincide, cf. \cite{Re1}. We
denote this $C^*$-algebra by $C^*\left(\mathcal
  G_{\varphi}\right)$.

Let $F : X \to \mathbb R$ be a continuous function. We can then define
a continuous homomorphism
$c_F : \mathcal G_{\varphi} \to \mathbb R$ such that
$$
c_F(x,k,y) = \lim_{n \to \infty} \left(\sum_{i=0}^{n+k} F\left(\varphi^i(x)\right) - \sum_{i=0}^n
F\left(\varphi^i(y)\right)\right) .
$$
%\begin{small}This amounts to the claim that
%\begin{equation}
%\begin{split}
%&(x,n-m,y) = (x,n'-m',y) \in \mathcal G(\varphi) \\
%& \Rightarrow \\
%&  \sum_{i=0}^n F\left(\varphi^i(x)\right) - \sum_{i=0}^m
%F\left(\varphi^i(y)\right) =  \sum_{i=0}^{n'} F\left(\varphi^i(x)\right) - \sum_{i=0}^{m'}
%F\left(\varphi^i(y)\right) 
%\end{split}
%\end{equation}
% \end{small}
%It is not difficult to see that any
%continuous homomorphism $c : \mathcal G_{\varphi} \to \mathbb R$ is of
%this form, cf. \cite{CT2}. 
There is a one-parameter group $\sigma^F$
on $C^*(\mathcal G_{\varphi})$ defined such that
$$
\sigma^F_t(f)(x,k,y) = e^{it c_F(x,k,y)} f(x,k,y)
$$
when $f \in C_c(\mathcal G_{\varphi})$, cf. \cite{Re1}. Let $\beta \in \mathbb R$. To describe the $\beta$-KMS-states of
$\sigma^F$ we say that a Borel measure $m$ on $X$ is
\emph{$e^{\beta  F}$-conformal} for $\varphi$ when 
\begin{equation}\label{conformal}
m\left(\varphi(A)\right) = \int_A e^{\beta F(x)} \ dm(x)
\end{equation} 
for every Borel subset $A \subseteq X$ such that $\varphi : A \to X$
is injective. This definition is due to Denker and Urbanski,
\cite{DU}, motivated by work of D. Sullivan. By Lemma 3.2 in \cite{Th2} the
condition is equivalent to $m$ being $(\mathcal G_{\varphi},
c_F)$-conformal with exponent $\beta$ as defined in \cite{Th1}. A $e^{\beta F}$-conformal
Borel probability measure $m$ on $X$ gives therefore rise to a $\beta$-KMS state
$\phi_m$ for $\sigma^F$ such that
\begin{equation}\label{no-atoms}
\phi_m(f) = \int_X f(x,0,x) \ dm(x) 
\end{equation}
when $f \in C_c\left(\mathcal G_{\varphi}\right)$. The calculation
necessary to confirm this was first performed by Renault, cf. Proposition 5.4
in \cite{Re1}. One
of the main achievements in \cite{N} is the realization that KMS states
of this sort do not exhaust them all.

Let $x \in X$ be a $\varphi$-periodic point such that
\begin{equation}\label{per}
\sum_{j=0}^{p-1} F\left(\varphi^j(x)\right) = 0
\end{equation}
when $p$ is the minimal period of $x$. Assume that $\beta \in \mathbb R$ is such that
\begin{equation}\label{sum}
 M = \sum_{n=1}^{\infty}  \ \sum_{y \in Y_n}  \exp \left( -\beta \sum_{j=0}^{n-1}
  F\left(\varphi^j(y)\right)\right) <  \infty ,
\end{equation}
where
$$
Y_n = \varphi^{-n}(x) \backslash \bigcup_{j=0}^{n-1} \varphi^{-j}(x).
$$ 
For $y\in X$, let $\delta_y$ be the Dirac probability measure at $y$. Then
$$
m_x = (1+M)^{-1} \left(\delta_x +\sum_{n=1}^{\infty}  \ \ \sum_{y \in
    Y_n}  \exp \left( -\beta \sum_{j=0}^{n-1}
  F\left(\varphi^j(y)\right)\right) \delta_y\right)
$$
is a Borel probability measure and it is easy to check, using (\ref{per}), that
$m_x$ is $e^{\beta F}$-conformal. For $\lambda \in \mathbb T$ and $y
\in \bigcup_{n=0}^{\infty} \varphi^{-n}(x)$, the map
$(y,kp,y) \mapsto \lambda^k$
is a character of the isotropy group
$$
\left(\mathcal G_{\varphi}\right)^y_y = \left\{ (y,kp,y): \ k \in \mathbb Z
\right\} .
$$ 
Since $m_x$ is supported on $ \bigcup_{n=0}^{\infty} \varphi^{-n}(x)$
we can consider this as a $m_x$-measurable field of states on the isotropy groups of $\mathcal
  G_{\varphi}$ in the sense of \cite{N}. It follows therefore from
  Theorem 1.1 in \cite{N} that there is a state $\phi^{\lambda}_x$ on
  $C^*(\mathcal G_{\varphi})$ such that
$$
\phi^{\lambda}_x(f) = \int_X \sum_{k \in \mathbb Z} \lambda^k
f(y,kp,y) \ dm_x(y) 
$$
for all $f \in C_c\left(\mathcal G_{\varphi}\right)$. Furthermore, it
follows from Theorem 1.3 in \cite{N} that $\phi^{\lambda}_x$ is a
$\beta$-KMS state. These facts follow also from the following lemma. To
formulate it, set $\Lambda = \bigcup_{n=0}^{\infty} \varphi^{-n}(x)$
and for $y\in \Lambda$ set
$$
|y| = \min \left\{ k \in \mathbb N: \varphi^k(y) = x\right\}.
$$
Define a one-parameter group $U_t, t\in
  \mathbb R$, of unitaries on $l^2\left( \Lambda\right)$ such that
\begin{equation}\label{unitary}
U_t\psi(y) = \exp \left(it\sum_{j=0}^{|y|-1}F\left(\varphi^j(y)\right)\right)\psi(y) 
\end{equation}
and a positive operator $H_{\beta}$ by 
\begin{equation}\label{b32p}
H_{\beta}\psi(y) = \exp \left( - \beta \sum_{j=0}^{|y|-1}
  F\left(\varphi^j(y)\right)\right)  \psi(y) .
\end{equation}
It follows from (\ref{sum}) that $H_{\beta}$ is trace-class. Note that $H_{\beta} = U_{i\beta}$.

\begin{lemma}\label{b31} For each $\lambda \in \mathbb T$ there is an irreducible $*$-representation
$\pi_{\lambda}$ of $C^*(\mathcal G_{\varphi})$ on $l^2\left(\Lambda\right)$ such that
\begin{equation}\label{equi}
\pi_{\lambda} \circ \sigma^F_t = \Ad U_t \circ \pi_{\lambda}
\end{equation}
for all $t \in \mathbb R$, and
\begin{equation*}\label{a23}
\phi^{\lambda}_{x}(a) = \frac{\Tr (H_{\beta}\pi_{\lambda}(a))}{\Tr(H_{\beta})},
\end{equation*}
for all $a \in C^*(\mathcal G_{\varphi})$.
\end{lemma}
\begin{proof} When $\varphi, \psi \in l^2(\Lambda)$ and $f\in
  C_c(\mathcal G_{\varphi})$ we have the estimate 
\begin{equation*}\label{a13}
\begin{split}
& \left(\sum_{y, z, k} |\varphi(y)| \left|f\left(y , |y| -
|z| + kp, z\right)\right| |\psi(z)| \right)^2 \\
& \leq \sum_{y, z, k} |\varphi(y)|^2 \left|f\left(y, |y| -
|z| + kp, z \right)\right| \sum_{y', z', k'} |\psi(z')|^2 \left|f\left(y' , |y'| -
|z'| + k'p, z'\right)\right| \\
& \leq \left\|f\right\|_I^2 \sum_{y} |\varphi(y)|^2
 \sum_{z} |\psi(z)|^2 ,
\end{split}
\end{equation*}
where $\left\|f\right\|_I$ is the $I$-norm of Renault,
cf. \cite{Re1}. Since $\|f\|_I < \infty$, it follows that
\begin{equation*}\label{a13}
\pi_{\lambda}(f)\psi(y) = \sum_{z\in \Lambda}\sum_{k \in \mathbb Z} \lambda^k f\left(y, |y| -|z| + kp, z\right)\psi(z) 
\end{equation*}
defines an element $\pi_{\lambda}(f)\psi \in
  l^2\left(\Lambda\right)$ when $\psi \in
  l^2\left(\Lambda\right)$. Furthermore, we see that $\pi_{\lambda}(f)$ is a bounded operator on $l^2(\Lambda)$
with $\left\|\pi_{\lambda}(f)\right\| \leq\|f\|_I$. It follows that
$\pi_{\lambda}$ extends to a map $\pi_{\lambda} :C^*\left(\mathcal
  G_{\varphi}\right) \to B\left(l^2\left(\Lambda\right)\right)$. It is
straightforward to verify that $\pi_{\lambda}$ is an irreducible
$*$-representation with the stated properties. We leave the details
with the reader who will then see how condition (\ref{per}) is
used to establish (\ref{equi}). 
\end{proof}

It follows that a $p$-periodic point $x$ of $\varphi$ for which
(\ref{per}) and (\ref{sum}) both hold gives rise to a $\beta$-KMS state
$\phi^{\mu}_x$ for every Borel probability measure $\mu$ on the circle
such that
$$
\phi^{\mu}_x = \int_{\mathbb T} \phi^{\lambda}_x \ d\mu(\lambda) .
$$  
In this way such a periodic orbit gives rise to a face in the simplex of
$\beta$-KMS states which is affinely homeomorphic to the Bauer simplex
of Borel probability measures on the circle. Note that the faces of
$\beta$-KMS states coming from different periodic points are the same when
the periodic points are in the same orbit and disjoint otherwise. Note
also that it follows from Lemma \ref{b31} that $\varphi^{\lambda}_x$
is an extremal $\beta$-KMS state, cf. Theorem 5.3.30 (3) in
\cite{BR}. If we assume that the $\varphi$-periodic points are
countable, and if $\beta \neq 0$, it follows from Theorem 2.4
in \cite{Th1} that the $\beta$-KMS state $\varphi_m$ of
(\ref{no-atoms}) arising from a \emph{continuous} (or non-atomic) $e^{\beta F}$-conformal Borel
probability measure $m$ is extremal if and only of $m$ is extremal in
the set of $e^{\beta  F}$-conformal Borel probability
measures. But note that the 'if'-part is not true for measures with
atoms. For example, the measure $m_x$ is extremal among the $e^{\beta
  F}$-conformal measures while
$$
\phi_{m_x} = \int_{\mathbb T} \phi^{\lambda}_x \ d\lambda
$$    
when we integrate with respect to Lebesgue measure $d\lambda$.

We have
now described the extremal KMS states that are relevant for the example considered in this paper, but for completeness we describe them all. The \emph{full orbit} $\mathcal O(z)$ of a point $z
\in X$ is the set
points $y \in X$ such that
\begin{equation}\label{ret1}
\varphi^n(z) = \varphi^m(y) 
\end{equation}
for some $n,m \in \mathbb N$. An element $z \in X$ is 
\emph{aperiodic} when $\mathcal O(z)$ does not contain a periodic
orbit. For such a point $z$ we can define
$\mathcal F : \mathcal O(z) \to \mathbb R$ such that
$$
\mathcal F (y) = \sum_{j=0}^{m-1} F\left(\varphi^j(y)\right) -
\sum_{j=0}^{n-1} F\left(\varphi^j(z)\right) ,
$$
where $n,m$ are numbers such that (\ref{ret1}) holds. For $\beta \in
\mathbb R$ we say that $z$ is \emph{$\beta$-summable} when
$$
M = \sum_{y \in \mathcal O(z)} e^{-\beta \mathcal F(y)} < \infty ,
$$
and we can then consider the Borel probability measure
$$
m_z = M^{-1} \sum_{y\in \mathcal O(z)} e^{-\beta \mathcal F(y)}
\delta_y .
$$ 
This is a $e^{\beta F}$-conformal measure and we can consider
the corresponding state $\phi_{m_z}$ which is an extremal
$\beta$-KMS state. By Theorem 1.3 in \cite{N} and Theorem 2.4 in \cite{Th1} we have now
described all extremal $\beta$-KMS states for
$\sigma^F$ and can summarize with the following 

\begin{thm}\label{Nesh} Assume that the periodic points of $\varphi$
  are countable. Let $\beta \in \mathbb R\backslash \{0\}$. The
  extremal $\beta$-KMS states for $\sigma^F$ are 
\begin{enumerate} 
\item[$\bullet$] the states $\phi_m$,
  where $m$ is an extremal and continuous $e^{\beta F}$-conformal
  Borel probability measure $m$ on $X$, 
\item[$\bullet$] the states
  $\varphi^{\lambda}_x$, where $\lambda \in \mathbb T$ and $x$ is a $p$-periodic
  point for $\varphi$ for which (\ref{per}) and (\ref{sum}) both
  hold, and
\item[$\bullet$] the states $\phi_{m_z}$, where $z\in X$ is aperiodic and $\beta$-summable.
\end{enumerate}
\end{thm}  

When $\varphi$ is surjective and does not change sign, in the sense
that either
$F(x) \geq 0$ for all $x \in X$ or $F(x) \leq 0$ for all $x \in X$, no aperiodic point can be $\beta$-summable for any $\beta$, and there
are then only the two first-mentioned classes of KMS-states to consider.

\section{Proof of the main result}

We apply Theorem \ref{Nesh} to the shift $\sigma$ on $\{0,1\}^{\mathbb
  N}$ and the function $F$ defined by (\ref{F}). For this note that
there is a $*$-isomorphism $O_2 \simeq C^*\left(\mathcal
  G_{\sigma}\right)$ sending the isometry $V_i \in O_2$ to the
characteristic function of the set $\left\{ (x,1,\sigma(x)) \in \mathcal G_{\sigma} : \ x_1 = i 
\right\}, \ i = 0,1$. Under this isomorphism the one-parameter group $\sigma^F$ from Section
\ref{sec1} is turned into the one-parameter group $\sigma^F$ from Section
\ref{sec2}.

\begin{lemma}\label{tuborg3} For each $\beta \in \mathbb R$ there is at
  most one $e^{\beta F}$-conformal probability measure for $\sigma$, and 
  none if $\beta < \beta_0$.
\end{lemma}
\begin{proof} For every word $u \in \mathbb \{0,1\}^n$ we let $[u]$ denote
  the corresponding cylinder set,
$$
[u] = \left\{ \left(x_i\right)_{i=1}^{\infty} \in \{0,1\}^{\mathbb N}:
  \ x_1x_2\cdots x_n = u \right\} .
$$
It follows from (\ref{conformal}) that a $e^{\beta F}$-conformal Borel
probability measure $\mu$
must satisfy 
$$
e^{\beta} \mu([0]) = \mu\left(\{0,1\}^{\mathbb N}\right) = 1,
$$
and hence $\mu([0]) = e^{-\beta}$ and $\mu([1]) = 1-e^{-\beta}$. Now
assume that we have determined $\mu([w])$ for every word $w \in
\{0,1\}^n$. Let $w$ be such a word. Then
$$
\mu([w]) = \int_{[0w]} e^{\beta F(x)} \ d\mu(x) = e^{\beta} \mu([0w]),
$$
and we conclude that $\mu([0w]) = e^{-\beta} \mu([w])$. Assume then
that $w$ does not consist entirely of $1$'s, and let $j$ be position
of the first $0$ occurring in $w$. Then
$$
\mu([w]) = \int_{[1w]} e^{\beta F(x)} \ d\mu(x) = e^{\frac{\beta}{j+1}} \mu([1w]),
$$
and hence $ \mu([1w]) =   e^{-\frac{\beta}{j+1}} \mu([w])$. In this
way we determine the value $\mu([u])$ for every word $u\in \{0,1\}^{n+1}$,
except the word consisting entirely of $1$'s. However, that number is
then determined by the condition that
$$
\sum_{u \in \ \{0,1\}^{n+1}} \mu([u]) = 1.
$$ 
This proves the uniqueness of a $e^{\beta F}$-conformal Borel
probability measure since regular Borel measures are determined by the
values they give to cylinder sets.

To prove that there can not be any $e^{\beta F}$-conformal probability
measure unless $\beta \geq \beta_0$, observe that the sets
$[0],[10],[110],[1110], \cdots $, are mutually disjoint, and that for any
Borel measure $\mu$ satisfying (\ref{conformal}) we have 
$$
\mu([1^{k-1}0]) = \exp\left(-\beta \sum_{j=1}^k
  \frac{1}{j}\right)\mu\left(\{0,1\}^{\mathbb N}\right) .
$$ 
If $\mu$ is a probability measure we must therefore have that
$
\sum_{k=1}^{\infty} \exp\left(-\beta \sum_{j=1}^k \frac{1}{j}\right) \leq 1$,
i.e. $\beta \geq \beta_0$.

\end{proof}

\begin{lemma}\label{sums} Let $\beta \in \mathbb R$. Then 
$$
\sum_{n=1}^{\infty} \ \sum_{ x \in 
    \sigma^{-n}\left(1^{\infty}\right) \backslash \sigma^{-n+1}\left(1^{\infty}\right)} \exp\left(-\beta \sum_{j=0}^{n-1}F(\sigma^j(x))\right) < \infty
$$ 
if and only if $\beta > \beta_0$.
\end{lemma}
\begin{proof} Set $Y_0 = \left\{1^{\infty}\right\}$ and $Y_n =
  \sigma^{-n}\left(1^{\infty}\right)  \backslash
  \sigma^{-n+1}\left(1^{\infty}\right), n \geq 1$. Then
\begin{equation}\label{sum2}
Y_n = 0Y_{n-1} \sqcup 10Y_{n-2} \sqcup 110Y_{n-3} \sqcup 1110Y_{n-4} \sqcup \cdots \sqcup 1^{n-1}0Y_0 . 
\end{equation}
Set 
$$
Z_n = \sum_{x \in Y_n}  \exp\left(-\beta
  \sum_{j=0}^{n}F(\sigma^j(x))\right) , 
$$
for $n \geq  0$, and
$s_k = 1 + \frac{1}{2} + \frac{1}{3} + \cdots + \frac{1}{k}$ for $k \geq 1$. It follows from (\ref{sum2}) that
\begin{equation}\label{sum5}
Z_n = e^{-\beta s_1} Z_{n-1} + e^{-\beta s_2} Z_{n-2} + \cdots +
e^{-\beta s_n} Z_0 
\end{equation}
for all $n \geq 1$, and then from (\ref{sum5}) that
\begin{equation}\label{sum6}
\sum_{n=1}^N Z_n \leq\left(\sum_{n=0}^N Z_n\right)\left(\sum_{k=1}^N
  e^{-\beta s_k}\right) \leq \sum_{n=1}^{2N} Z_n 
\end{equation}
for all $N \geq 1$. It is straightforward to deduce from (\ref{sum6}) that
$\sum_{k=1}^{\infty} e^{- \beta s_k} < 1$ if and only if $\sum_{n=1}^{\infty} Z_n <
\infty$. This completes the proof.
\end{proof}

Except for the existence of a $\beta_0$-KMS state, Theorem \ref{MAIN1} follows now by combining Lemma \ref{tuborg3} and
Lemma \ref{sums} with Theorem \ref{Nesh}. That there \emph{is} a
$\beta_0$-KMS state follows from the fact that set of $\beta$'s for
which there is a $\beta$-KMS state is closed by Proposition 5.3.25 in \cite{BR}.

\section{Factor types}\label{factortypes}

\begin{thm}\label{types} For $\beta > \beta_0$ the von Neumann algebra
  factor generated by the GNS representation of an extremal $\beta$-KMS state for $\sigma^F$ is
  of type $I_{\infty}$. The von Neumann algebra generated by the GNS
  representation of the unique $\beta_0$-KMS state for $\sigma^F$ is
  the hyper-finite $III_1$-factor.
\end{thm}
\begin{proof} The assertion concerning the cases when $\beta > \beta_0$ follows from Lemma \ref{b31}. Let $\omega$ be the $\beta_0$-KMS
  state. It follows from
  Theorem \ref{Nesh} and Lemma \ref{sums} that there is a continuous
  Borel probability measure $\mu$ on $\{0,1\}^{\mathbb N}$ such that
$$
\omega(f) = \int_{\{0,1\}^{\mathbb N}} f(x,0,x) \ d\mu(x)
$$
for $f \in C_c\left(\mathcal G_{\sigma}\right)$. Let $\pi_{\omega}$ be
the GNS-representation of $\omega$ and set $M
  = \pi_{\omega}\left(C^*(\mathcal G_{\sigma})\right)''$. Then $M$ is
  a $\sigma$-finite injective factor since $C^*\left(\mathcal G_{\sigma}\right)$ is
 separable and nuclear, and $\omega$ is an extremal $\beta_0$-KMS state. By
  Haagerups result, \cite{Ha}, it suffices therefore now to show that
  $\Gamma(M) = \mathbb R$ where $\Gamma(M)$ is the invariant from
  Connes' thesis, \cite{C}. This will be achieved via an elaboration of
  the method used for the proof of Proposition 4.11 in \cite{Th2} and
  Theorem 3.2 in \cite{Th3}. Let $\tilde{\omega}$
  be the faithful normal state on $M$ extending $\omega$ in the sense that $\tilde{\omega} \circ
  \pi_{\omega} = \omega$. The modular automorphism group $\theta$ of $M$
  corresponding to $\tilde{\omega}$ is $\theta_t =
  \tilde{\sigma}^F_{\beta_0 t}$, where $\tilde{\sigma}^F$ is the $\sigma$-weakly continuous
  one-parameter group such that $\tilde{\sigma}^F_t
  \circ \pi_{\omega} = \pi_{\omega} \circ \sigma^F_t$ for all $t$,
  cf. Theorem 8.14.5 in \cite{Pe}. It follows from \cite{C} that
$$
\Gamma(M) = \bigcap_q \Sp(qMq),
$$
where we take the intersection over all non-zero central projections
$q$ in the fixed point algebra $N$ of $\theta$, and
$\Sp(qMq)$ is the Arveson spectrum of the restriction of
$\theta$ to $qMq$. Since $\Gamma(M)$ is a closed subgroup of $\mathbb R$ it
suffices to take a $K \in \mathbb N, \ K \geq 2$, and a non-zero central
projection $q$ in $N$ and show that
\begin{equation}\label{conclusion}
\frac{\beta_0}{K+1} \in \Sp(qMq).
\end{equation}

% show that 
%\begin{equation}\label{IN}
%q \in \pi_{\omega}\left(C\left(
%    \{0,1\}^{\mathbb N}\right)\right)''.
%\end{equation} 
In the following we identify $\{0,1\}^{\mathbb N}$ with the unit space
of $\mathcal G_{\sigma}$ via the embedding $x \mapsto (x,0,x)\in
\mathcal G_{\sigma}$, and we set $\left\|a\right\|_{\omega} =
\sqrt{\tilde{\omega}(a^*a)}$ for $a \in M$. Furthermore, to simplify
notation we suppress $\pi_{\omega}$ from the notation and consider
$C_c\left(\mathcal G_{\sigma}\right)$ as a subalgebra of $M$. Finally,
we use $r$ and $s$ to denote the range and source map of $\mathcal
G_{\sigma}$, i.e. $r(x,k,y) = x$ and $s(x,k,y) =y$.

To approximate $q$ with
elements from $C\left(\{0,1\}^{\mathbb
      N}\right)$ we first prove
\begin{assert}\label{assert1}
For every $f \in C_c\left(\mathcal G_{\sigma}\right)$ and every
$\epsilon > 0$ there is a finite set $d_j,j =0,1,2, \cdots, m$, of
elements in $C\left(\{0,1\}^{\mathbb N}\right)$ such that
$\sum_{j=0}^m d_j^2 = 1$ and
$\sum_{j=0}^m d_jfd_j = f|_{\{0,1\}^{\mathbb N}} + h$ where
$\left\|h\right\|_{\omega} \leq \epsilon$. 
\end{assert}

Set $\Is \mathcal G_{\sigma} = \left\{ (x,k,x) \in \mathcal
  G_{\sigma}: \ x \in \{0,1\}^{\mathbb N} , \ k \neq 0
\right\}$. Write $f = f|_{\{0,1\}^{\mathbb N}}
+ f_1$ where $\supp f_1
\subseteq \mathcal G_{\sigma} \backslash \{0,1\}^{\mathbb
  N}$. Consider an open bisection in $\mathcal G_{\sigma}$ of the form
$$
\mathcal U =\left\{ (x,n-m,y): \ \sigma^n (x) = \sigma^m(y) , \ x\in W, \ y \in V
\right\}
$$
with $\sigma^n$ injective on $W$ and $\sigma^m$ injective on
$V$, and $n \neq m$. An element $\mathcal U \cap \Is \mathcal G_{\sigma}$ must have
the form $(z,n-m,z)$ where $\sigma^{\min \{n,m\}}(z)$ is
$|n-m|$-periodic for $\sigma$. There are only finitely many such
elements $z$ and we conclude therefore that $\supp f_1 \cap \Is \mathcal
G_{\sigma}$ is a finite set. Since $\mu$ is continuous,
$
\mu\left( r \left(\supp f_1 \cap \Is \mathcal
G_{\sigma}\right)\right) = 0$ and we can choose an open neighborhood
$U$ of $ r \left(\supp f_1 \cap \Is \mathcal
G_{\sigma}\right)$ such that 
\begin{equation}\label{esti}
\mu(U) \leq \frac{\epsilon^2}{4\|f\|^2+1}.
\end{equation}
Write $f_1 = f_2 +
f_3$ where $f_2 \in C_c\left( r^{-1}(U) \cap s^{-1}(U)\right)$ and
$\supp f_3 \subseteq \mathcal G_{\sigma} \backslash \left(\{0,1\}^{\mathbb
  N} \cup \Is \mathcal G_{\sigma}\right)$. For each $x \in
r\left(\supp f_3\right)$ there is an open neighborhood $U_x$ of $x$
such that 
$$
s \left( r^{-1}(U_x) \cap \supp f_3\right) \cap U_x =
\emptyset.
$$ 
%(Indeed, if not there is for each neighbohood $W$ of $x$
%an element $z_w \in \supp f_3$ such that $r(z_w) \in W$ and $s(r_w)
%\in W$. By compactness of $\supp f_3$ this will give us an element $z
%\in \supp f_3$ such that $s(z) = r(z) = x$, contradicting that $\supp f_3 \subseteq \mathcal G_{\sigma} \backslash \left(\{0,1\}^{\mathbb
%  N} \cup \Is \mathcal G_{\sigma}\right)$.) 
It follows that there is a
finite open cover $U_i, i = 1,2, \cdots, m$, of $r\left(\supp
  f_3\right)$ such that $s \left( r^{-1}(U_i) \cap \supp f_3\right) \cap U_i =
\emptyset$ for each $i$. Let $k_j, j = 0,1,2, \cdots,m$, be a
partition of unity in $C\left(\{0,1\}^{\mathbb N}\right)$ such that
$\supp k_j \subseteq U_j, j =1,2,\cdots, m$, and $\supp k_0 \cap
r\left( \supp f_3\right) = \emptyset$. Set $d_j = \sqrt{k_j}$. Then $ \sum_{j=0}^m
d_j f_3d_j = 0$ and hence
$$
\sum_{j=0}^m d_j fd_j =  f|_{\{0,1\}^{\mathbb N}} + \sum_{j=0}^m d_j
f_2d_j .
$$
Set $h = \sum_{j=0}^m d_j
f_2d_j$. Then $\left\|h\right\| \leq 2 \|f\|$ and $\supp h \subseteq
r^{-1}(U) \cap s^{-1}(U)$. In particular, from the last fact it
follows that $h^*h|_{\{0,1\}^{\mathbb N}}$ is supported in $U$ and we
find therefore from (\ref{esti}) that 
$$
\left\| h\right\|^2_{\omega} = 
  \int_{\{0,1\}^{\mathbb N}} h^*h|_{\{0,1\}^{\mathbb N}} \ d\mu \leq \|h\|^2
  \mu (U) \leq \epsilon^2.
$$
This proves Assertion \ref{assert1}. To use it, note that an
application of Kaplansky's density theorem gives us a sequence
$\{g_n\}$ of self-adjoint elements in $C_c(\mathcal G_{\sigma})$ such
that $\|g_n\| \leq 1$ for all $n$ and $\lim_{n \to \infty}
\left\|q - g_n \right\|_{\omega} = 0$. Set $f_n = g_n|_{\{0,1\}^{\mathbb
    N}}$ and let $\epsilon >0$. Let $\{d_j\}_{j=0}^m$ be a partition
of unity  in
$C\left(\{0,1\}^{\mathbb N}\right)$ obtained by
applying Assertion \ref{assert1} with $f = g_n$. Using Kadisons
Schwarz-inequality for the second inequality and the KMS-property for
the second equality we find that
\begin{equation*}
\begin{split}
& \left\| q- f_n\right\|_{\omega} - \epsilon \ \leq \ \left\|
  q- \sum_{j=0}^m d_jg_nd_j\right\|_{\omega}  = \left\| \sum_{j=0}^m d_j (q -
    g_n)d_j \right\|_{\omega} \\\
& \leq
\ \sqrt{\tilde{\omega} \left( \sum_{j=0}^m d_j (q -
    g_n)^2 
    d_j\right)} = \sqrt{\tilde{\omega} \left( \sum_{j=0}^m  \left(q -
    g_n\right)^2 
    d_j^2\right)}  = \left\| q- g_n\right\|_{\omega}.
\end{split}
\end{equation*} 
It follows that $\lim_{n \to \infty}
\left\|q - f_n \right\|_{\omega} = \lim_{n \to \infty}
\left\|q - g_n \right\|_{\omega} = 0$. 
%To complete the
%proof of (\ref{IN}) it remains only to use this conclusion to show
%that $\lim_{n \to \infty} \pi_{\omega}(f_n) = q$ in the strong
%operator topology. Since all the operators involved have their norms
%dominated by $1$ it suffices for this to observe that when $g \in
%C_c\left(\mathcal G_{\sigma}\right)$ we have that
%\begin{equation}
%\begin{split}
%&\tilde{\omega}\left( \pi_{\omega}(g)^* \left( q-
%    \pi_{\omega}(f_n)\right)^2\pi_{\omega}(g)\right) \leq 2  \tilde{\omega}\left( \pi_{\omega}(g)^* \left( q-
%    \pi_{\omega}(f_n)\right)\pi_{\omega}(g)\right)\\
%& = 2  \tilde{\omega}\left( \left( q-
%    \pi_{\omega}(f_n)\right)\pi_{\omega}(g)\pi_{\omega}\left(\sigma^F_{i\beta_0}(g^*)\right)\right) \\
%&\leq 2\left\|q - \pi_{\omega}(f_n) \right\|_{\omega} \left\| g \sigma^F_{i\beta_0}(g^*)\right\| . 
%\end{split}
%\end{equation}
Since finite linear combinations of $\left\{1_{[u]} : \ u \in \bigcup_n
  \{0,1\}^n \right\}$ is a norm-dense $*$-subalgebra of
$C\left(\{0,1\}^{\mathbb N}\right)$ we may assume that each $g_n$ is
such a linear combination. Since $q$ is a projection a standard argument, as in
the proof of Lemma 12.2.3 in \cite{KR}, shows that for every $\epsilon
> 0$ there is a finite collection $u_1,u_2, \cdots, u_L \in  \bigcup_n
  \{0,1\}^n $ such that $[u_i] \cap [u_j] = \emptyset$ when $i \neq
  j$, and 
\begin{equation}\label{cruxest}
\left\| q - p\right\|_{\omega} < \epsilon,
\end{equation} 
where $\epsilon > 0$ is as small as we need, and $p = \sum_{i=1}^L 1_{[u_i]}$. We
choose $\epsilon > 0$ so small that
$$
  \exp\left(-\beta_0 \left(\sum_{j=1}^{K} \frac{1}{j}\right)\right)
\epsilon  +\exp \left(\frac{\beta_0}{K}\right) \epsilon + 2\epsilon
\ < \
\exp\left(-\beta_0 \left(\sum_{j=1}^{K+1} \frac{1}{j}\right)\right)
\tilde{\omega}(q).
%e^{-F(l^+)\beta}
%\epsilon  +e^{-(F(l^+)-F(l^-))\beta} \epsilon + 2\epsilon  \ <  \ e^{-F(l^+)\beta}
%\omega(q).
$$
%We use here that $\tilde{\omega}$ is faithful on $M$ so that
%$\tilde{\omega}(q) > 0$, cf. Corollary 5.3.9 in \cite{BR}. 
Fix an $i$ and consider the mutually disjoint sets $\left[u_i0\right],
\left[u_i10\right], \left[u_i110\right], \left[u_i1110\right], \
\cdots$. Since $[u_i] \backslash \bigcup_{j=0}^{\infty} \left[
  u_i1^j0\right]$ only contains one element, and since $\mu$ does not
have atoms, we conclude that 
$$
\lim_{m \to \infty} \left\| 1_{[u_i]} - \sum_{j=0}^m
 1_{ \left[u_i1^j0\right]} \right\|_{\omega} = 0.
$$
Exchanging $u_i$ with $u_i1^j0, j = 0,1,2, \cdots, n$, for some
sufficiently large $n$, and doing a similar thing with the other
$u_i$'s, we can arrange that each $u_i$ terminates with $0$. %It follows
%from this and the condition of conformality for the function $F$ that
%\begin{equation}\label{NN}
%\mu\left( \left[ u_i 1^{N-1}0\right]\right) = \mu([u_i])\mu\left(
%    \left[1^{N-1}0\right]\right) = \exp
%  \left(-\beta_0\left(\sum_{j=1}^N \frac{1}{j}\right)\right) \mu([u_i])
%\end{equation}
%for all $N \in \mathbb N$. 
Set $l^+ = 1^{K-1}0$, $l^- = 1^K0$ and $\left|l^+\right| =K,
\left|l^-\right| = K+1$. For each $i$ we let $w^{\pm}_i \in C_c(\mathcal G_{\sigma})$ be the
characteristic function of the compact and open set
$$
\left\{ ( u_ix, - |l^{\pm}|, u_il^{\pm}x) : \ x \in \{0,1\}^{\mathbb N}
\right\}
$$
in $\mathcal G_{\sigma}$. Each $w^{\pm}_i$ is a partial isometry such that
\begin{enumerate}
\item[a)] $w^{\pm}_i{w^{\pm}_i}^* = 1_{[u_i]}$ ,
\item[b)] ${w^{\pm}_i}^* w^{\pm}_i = 1_{[u_il^{\pm}]} \leq
  1_{[u_i]}$, and
\item[c)] $\sigma^F_t\left(w^{+}_i\right) = \exp \left( -it
    \left(\sum_{j=1}^K \frac{1}{j}\right)\right)w^{+}_i$ while $\sigma^F_t\left(w^{-}_i\right) = \exp \left( -it
    \left(\sum_{j=1}^{K+1} \frac{1}{j}\right)\right)w^{-}_i$ 
  for all $t \in \mathbb R$.
\end{enumerate} 
By using that $u_i$ terminates with $0$, and that $l^- = 1^K0$, it
follows from the conformality condition (\ref{conformal}) and the
definition of $F$ that% It follows from (\ref{NN}) that
\begin{equation}\label{d)}
\begin{split}
%&\mu\left(\left[u_il^{+})\right]\right) = \exp \left( -
%        \beta_0\left(\sum_{j=1}^K \frac{1}{j}\right)\right) \mu
%      ([u_i]),   \\
&\mu\left(\left[u_il^{-}\right]\right) = \exp \left( -
        \beta_0\left(\sum_{j=1}^{K+1} \frac{1}{j}\right)\right) \mu
      ([u_i]).
\end{split}
\end{equation} 
Set $w_{\pm} = \sum_{i=1}^L w^{\pm}_i$ and note that $w_{\pm}$ are partial isometries. It follows from b) that $w_+p = w_+$ and therefore from
(\ref{cruxest}) that
\begin{equation*}
\begin{split}
& \tilde{\omega} (q{w_+}^*w_-q{w_-}^*w_+q) \geq
\tilde{\omega}({w_+}^*w_-q{w_-}^*w_+) - 2\epsilon .
\end{split}
\end{equation*}
Since $\tilde{\omega}$ is a $\beta_0$-KMS state for $\tilde{\sigma}^F$, it follows from c) that
$$
\tilde{\omega}({w_+}^*w_-q{w_-}^*w_+) = e^{\frac{\beta_0}{K+1}}
\tilde{\omega}(q{w_-}^*w_+{w_+}^*w_-), 
$$
which thanks to a) is the same as
$e^{\frac{\beta_0}{K+1}}\tilde{\omega}(q{w_-}^*w_-)$. Using (\ref{cruxest}) again we find that
$$
\tilde{\omega}(q{w_-}^*w_-) \geq \tilde{\omega}(p{w_-}^*w_-) - \epsilon .
$$
It follows from b) that $ \tilde{\omega}(p{w_-}^*w_-) =
\tilde{\omega}({w_-}^*w_-)$ while b), a) and (\ref{d)}) imply that
$$
\tilde{\omega}({w_-}^*w_-) = \mu\left(\cup_i \left[u_i l^-\right]\right) =
\exp\left(-\beta_0 \left(\sum_{j=1}^{K+1} \frac{1}{j}\right)\right) \tilde{\omega}(p). 
$$
Since $\tilde{\omega}(p) \geq \tilde{\omega}(q) - \epsilon$ by (\ref{cruxest}) we can
put everything together and conclude that
\begin{equation*}
\begin{split}
&\tilde{\omega}(q{w_+}^*w_-q{w_-}^*w_+q) \geq \\
& \exp\left(-\beta_0 \left(\sum_{j=1}^{K} \frac{1}{j}\right)\right)
\tilde{\omega(q)} - \exp\left(-\beta_0 \left(\sum_{j=1}^{K} \frac{1}{j}\right)\right)
\epsilon  -\exp \left(\frac{\beta_0}{K+1}\right) \epsilon -2\epsilon .
\end{split}
\end{equation*}
It follows therefore from the choice of $\epsilon$ that
$\tilde{\omega}(q{w_+}^*w_-q{w_-}^*w_+q) > 0$. By c), 
$$
\tilde{\sigma}^F_t(q{w_+}^*w_-q) = e^{-i t\frac{1}{K+1}} q{w_+}^*w_-q
$$ 
and hence $\theta_{t} (q{w_+}^*w_-q) =  e^{-i t\frac{\beta_0}{K+1}} 
q{w_+}^*w_-q$ for all $t \in \mathbb R$.
Since $q{w_+}^*w_-q \neq 0$
we obtain the desired conclusion (\ref{conclusion}) from Lemme 2.3.6 in \cite{C}.

\end{proof}

\end{document}